\documentclass[reqno]{amsart}
\usepackage{amsmath}
\usepackage{amsfonts}
\usepackage{amssymb,enumerate}
\usepackage{amsthm}
\usepackage[all]{xy}
\usepackage{rotating}
\usepackage{hyperref}
\theoremstyle{plain}
\newtheorem{lem}{Lemma}

\newtheorem{prop}[lem]{Proposition}
\newtheorem{thm}[lem]{Theorem}

\theoremstyle{definition}
\newtheorem{defn}[lem]{Definition}
\newtheorem{ex}[lem]{Example}

\newtheorem{rmk}[lem]{Remark}

\newtheorem{notn}[lem]{Notation}
\newtheorem{fact}[lem]{Fact}

\newtheorem*{Convention}{Convention}

\newtheorem{Proof}[lem]{Proof}

\newcommand{\cat}[1]{\mathcal{#1}}

\newcommand{\catd}{\cat{D}}

\newcommand{\depth}{\operatorname{depth}}

\newcommand{\ann}{\operatorname{Ann}}

\newcommand{\HH}{\operatorname{H}}
\newcommand{\Hom}{\operatorname{Hom}}	

\newcommand{\spec}{\operatorname{Spec}}

\newcommand{\shift}{\mathsf{\Sigma}}

\newcommand{\cone}{\operatorname{Cone}}

\newcommand{\ideal}[1]{\mathfrak{#1}}
\newcommand{\m}{\ideal{m}}
\newcommand{\n}{\ideal{n}}
\newcommand{\p}{\ideal{p}}

\newcommand{\comp}[1]{\widehat{#1}}

\newcommand{\supp}{\operatorname{Supp}}

\newcommand{\Min}{\operatorname{Min}}

\newcommand{\bbz}{\mathbb{Z}}

\newcommand{\xra}{\xrightarrow}

\newcommand{\vf}{\varphi}

\newcommand{\x}{\mathbf{x}}

\renewcommand{\geq}{\geqslant}
\renewcommand{\leq}{\leqslant}

\newcommand{\Ext}[4][R]{\operatorname{Ext}_{#1}^{#2}(#3,#4)}	
\newcommand{\Rhom}[3][R]{\mathbf{R}\!\operatorname{Hom}_{#1}(#2,#3)}	
\newcommand{\Lotimes}[3][R]{#2\otimes^{\mathbf{L}}_{#1}#3}
\newcommand{\Otimes}[3][R]{#2\otimes_{#1}#3}
\renewcommand{\Hom}[3][R]{\operatorname{Hom}_{#1}(#2,#3)}	
\newcommand{\Tor}[4][R]{\operatorname{Tor}^{#1}_{#2}(#3,#4)}

\newcommand{\nak}{\operatorname{NAK}}

\numberwithin{equation}{lem}

\begin{document}

\bibliographystyle{amsplain}

\author{Sean Sather-Wagstaff}

\address{Department of Mathematics,
NDSU Dept \# 2750,
PO Box 6050,
Fargo, ND 58108-6050
USA}

\email{sean.sather-wagstaff@ndsu.edu}

\urladdr{http://www.ndsu.edu/pubweb/\~{}ssatherw/}

\thanks{The author was supported in part by a grant from the NSA}

\title{Ascent properties for derived functors}

%\date{\today}

%\dedicatory{}


\begin{abstract}
Given a flat local ring homomorphism $R\to S$, 
and two finitely generated $R$-modules $M$ and $N$, we describe conditions under which the modules
$\Tor iMN$ and $\Ext iMN$ have $S$-module structures that are compatible with their $R$-module structures.
\end{abstract}

\maketitle

%\tableofcontents

%\section{Introduction} \label{sec130216a}

\begin{Convention}
Throughout this paper, the term ``ring'' is short for ``commutative noetherian ring with identity'',
and ``module'' means ``unital module''.
\end{Convention}

We are interested in ascent of module structures along certain ring homomorphisms, following~\cite{anderson:neams,frankild:dcev,frankild:amsveem}
where the following result is proved; see~\cite[Theorems 1.5 and 1.7]{anderson:neams}.
Note that the natural maps from $R$ to its completion $\comp R$ and to its henselization $R^{\text{h}}$ satisfy the
hypotheses of this result.

\begin{fact}
\label{fact130216a}
Let $\vf\colon (R,\m,k)\to (S,\n,l)$ be a flat local ring homomorphism such that the induced map $R/\m\to S/\m S$ is an isomorphism,
and let $N$ be a finitely generated $R$-module.
Then the following conditions are equivalent:
\begin{enumerate}[\rm(i)]
\item \label{fact130216a1}
$N$ has an $S$-module structure compatible with its $R$-module structure via $\vf$.
\item \label{fact130216a2}
The natural map $\Hom SN\to N$ given by $f\mapsto f(1)$ is an isomorphism.
\item \label{fact130216a3}
The natural map $N\to \Otimes SN$ given by $n\mapsto \Otimes[]1n$ is an isomorphism.
\item \label{fact130216a4}
$\Ext iSN=0$ for all $i\geq 1$.
\item \label{fact130216a5}
$\Ext iSN$ is finitely generated over $S$ (or over $R$) for  $i= 1,\ldots,\dim_R(N)$.
\item \label{fact130216a6}
$\Otimes SN$ is finitely generated over $R$.
\item \label{fact130216a7}
The induced map $R/\ann_R(N)\to S/\ann_R(N)S$ is an isomorphism.
\item \label{fact130216a8}
For all $\p\in\Min_R(N)$ (equivalently, for all $\p\in\supp_R(N)$), the induced map $R/\p\to S/\p S$ is an isomorphism.
\end{enumerate}
\end{fact}

%\begin{rmk}\label{rmk130217b}
Condition~\eqref{fact130216a8} in  this fact shows that ascent (i.e., condition~\eqref{fact130216a1}) is somehow a topological condition
on the closed set $\supp_R(N)\subseteq\spec(R)$. 
%\end{rmk}
The point of this note is to exploit this idea to identify conditions on $M$ and $N$
that guaranteed ascent of module structures for $\Ext iMN$ and $\Tor iMN$. 
%The equivalent conditions in  Theorem~\ref{thm130216a} are satisfied 
Of course, one way to guarantee that $\Ext iMN$ and $\Tor iMN$ have compatible $S$-module structures is for
$M$ or $N$ to have a compatible $S$-module. For instance, if $M$ has such an $S$-module structure,
then so does $\Otimes MN$ by the formula $s(m\otimes n):=(sm)\otimes n$.
However, straightforward examples show that this condition is sufficient but not necessary.

\begin{ex}\label{ex130217a}
Let $k$ be a field, and consider the localized polynomial ring $R=k[X,Y]_{(X,Y)}$.
The modules $R/XR$ and $R/YR$ do not have $\comp R$-module structures compatible with their $R$-module structures,
by Fact~\ref{fact130216a}. However, the modules $\Ext i{R/XR}{R/YR}$ and $\Tor i{R/XR}{R/YR}$ are finite-dimensional
vector spaces over $R/(X,Y)R\cong \comp R/(X,Y)\comp R$, so they do have compatible $\comp R$-module structures.
\end{ex}

Our main result is the following.

\begin{thm}\label{thm130216a}
Let $\vf\colon (R,\m,k)\to (S,\n,l)$ be a flat local ring homomorphism such that the induced map $R/\m\to S/\m S$ is an isomorphism,
and let $M$ and $N$ be finitely generated $R$-modules.
Then the following conditions  are equivalent:
\begin{enumerate}[\rm(i)]
\item \label{thm130216a1}
$\Otimes MN$ has an $S$-module structure compatible with its $R$-module structure via $\vf$.
\item \label{thm130216a2}
$\Tor iMN$ has an $S$-module structure compatible with its $R$-module structure via $\vf$ for all $i\geq 0$.
\item \label{thm130216a3}
$\Ext iMN$ has an $S$-module structure compatible with its $R$-module structure via $\vf$  for all $i\geq 0$.
\item \label{thm130216a4}
$\Ext iMN$ has an $S$-module structure compatible with its $R$-module structure via $\vf$ for $i=0,\ldots,\dim_R(N)-1$.
\item \label{thm130216a5}
$\Ext i{\Otimes SM}N$ is finitely generated over $R$ for all $i\geq 1$.
\item \label{thm130216a6}
The natural map $\Ext i{\Otimes SM}N\to\Ext iMN$ is bijective for all $i\geq 0$.
\item \label{thm130216a7}
For all $\p\in\supp_R(M)\cap \supp_R(N)$ (equivalently, for all $\p$ that are minimal elements of $\supp_R(M)\cap \supp_R(N)$), the induced map $R/\p\to S/\p S$ is an isomorphism.
\end{enumerate}
\end{thm}

\begin{rmk}\label{rmk130217a}
The special case $M=R$ in Theorem~\ref{thm130216a} recovers much of Fact~\ref{fact130216a}. Of course, we use Fact~\ref{fact130216a}
in the proof of Theorem~\ref{thm130216a}, so we are not claiming that the fact is a corollary of the theorem.

One can combine Fact~\ref{fact130216a}
with Theorem~\ref{thm130216a} in several ways to give other  conditions equivalent to the ones from Theorem~\ref{thm130216a},
like the following:
\begin{enumerate}[\rm(i)]
\item[(ii')]
The natural map $\Hom S{\Otimes MN}\to \Otimes MN$ given by $f\mapsto f(1)$ is an isomorphism.
\item[(vi')] 
$\Otimes S{\Otimes MN}$ is finitely generated over $R$.
\end{enumerate}
We leave other such variations to the interested reader. See, though, Proposition~\ref{prop130217a}.

Regarding the range $i\geq 1$ in condition~\eqref{thm130216a5}, note that the module
$$\Ext 0{\Otimes SM}N\cong\Hom{\Otimes SM}N\cong\Hom S{\Hom MN}$$
is automatically finitely generated by~\cite[Corollary 1.7]{frankild:amsveem}.
\end{rmk}

The proof of Theorem~\ref{thm130216a}  is given in Proof~\ref{proof130217a} below, after a few preliminaries.

\begin{notn}\label{notn130217a}
Because it is convenient for us, we use some standard notions from the derived category $\catd(R)$ of
the category of $R$-modules~\cite{hartshorne:rad, verdier:cd, verdier:1}, with some notation that is summarized  in~\cite{christensen:gd}. 
In particular, $R$-complexes are indexed homologically
$$X=\cdots\to X_i\to X_{i-1}\to\cdots$$
and the \emph{supremum} of an $R$-complex $X$ is
$$\sup(X):=\sup\{i\in\bbz\mid\HH_i(X)\neq 0\}.$$
Given two $R$-complexes $X$ and $Y$, the derived Hom-complex and derived tensor product of $X$ and $Y$ are denoted
$\Rhom XY$ and $\Lotimes XY$, respectively.

We use the term ``morphism of complexes'' (also known as ``chain map'') for a morphism in the category of $R$-complexes.
A \emph{quasiisomorphism} is a morphism of complexes such that each induced map on homology modules is an isomorphism.
The complex $\shift^nX$ is obtained by shifting $X$ by $n$ steps to the left.
\end{notn}

The next two lemmas are almost certainly standard, but we include proofs since they are relatively straightforward.

\begin{lem}\label{lem130216a}
Let $(R,\m,k)$ be a local ring, and let $M$, $N$ be non-zero finitely generated $R$-modules.
Then there is an integer $i\leq\depth_R(N)$ such that $\Ext iMN\neq 0$. 
\end{lem}

\begin{proof}
Let $K$ denote the Koszul complex on a system of parameters for $R$.
We use the depth-sensitivity of $K$, which states that $\dim(R)-\sup(\Lotimes KN)=\depth_R(N)$. 
Since the homologies of $\Lotimes KN$ have finite length and $M$ is a non-zero module, it follows that
\begin{align*}
\sup(\Rhom M{\Lotimes KN})
&=\sup(\Lotimes KN)
=\dim(R)-\depth_R(N).
\end{align*}
The tensor-evaluation isomorphism
$\Rhom M{\Lotimes KN}\simeq\Lotimes K{\Rhom MN}$ implies that
$$
\dim(R)-\depth_R(N)
=\sup(\Lotimes K{\Rhom MN})
\leq\dim(R)+\sup(\Rhom MN)
$$
from which we conclude that $\sup(\Rhom MN)\geq-\depth_R(N)$.
The desired conclusion now follows. 
\end{proof}

\begin{lem}\label{lem130217a}
Let $R$ be a ring, and let $f\colon X\to Y$ be morphism of $R$-complexes. Assume that $\HH_i(X)$ and $\HH_i(Y)$ are finitely generated
over $R$ for all $i$.
Let $\x=x_1,\ldots,x_n$ be a sequence in the Jacobson radical of $R$,
and let $K$ be the Koszul complex $K^R(\x)$.
Then $f$ is a quasiisomorphism if and only if $\Otimes Kf\colon \Otimes KX\to \Otimes KY$ is a quasiisomorphism.
\end{lem}

\begin{proof}
The forward implication follows from the fact that $K$ is a bounded (below) complex of flat $R$-modules. 
For the converse, assume that $\Otimes Kf$ is a quasiisomorphism. 
It follows that the mapping cone $\cone(\Otimes Kf)\cong\Otimes K{\cone(f)}$ is exact,
and it suffices to show that $\cone(f)$ is exact. This is done by induction on $n$, the length of the sequence $\x$;
this reduces immediately to the case $n=1$.
Part of the long exact sequence for $\cone(f)$ and $\Otimes {K^R(x_1)}{\cone(f)}$
has the following form.
$$\HH_i(\cone(f))\xra{x_1}\HH_i(\cone(f))\to\underbrace{\HH_i(\Otimes {K^R(x_1)}{\cone(f)})}_{=0}$$
Since the homology module $\HH_i(\cone(f))$ is finitely generated,  Nakayama's Lemma implies that $\HH_i(\cone(f))=0$.
\end{proof}

\begin{Proof}[Proof of Theorem~\ref{thm130216a}]
\label{proof130217a}
The implications \eqref{thm130216a2}$\implies$\eqref{thm130216a1}, \eqref{thm130216a3}$\implies$\eqref{thm130216a4},
and \eqref{thm130216a6}$\implies$\eqref{thm130216a3} are routine. 
The implications \eqref{thm130216a1}$\implies$\eqref{thm130216a7}$\implies$\eqref{thm130216a2}
and \eqref{thm130216a7}$\implies$\eqref{thm130216a3} follow from Fact~\ref{fact130216a},
since 
$\supp_R(\Ext iMN)$ and $\supp_R(\Tor iMN)$
are contained in $\supp_R(M)\cap \supp_R(N)
=\supp_R(\Otimes MN)$.

%$\eqref{thm130216a4}\implies\eqref{thm130216a7}$.
\eqref{thm130216a4}$\implies$\eqref{thm130216a7}.
Assume that $\Ext iMN$ has an $S$-module structure compatible with its $R$-module structure via $\vf$ for $i=0,\ldots,\dim_R(N)-1$.
Fix a prime $\p\in\supp_R(M)\cap\supp_R(N)$,
and suppose that the induced map $R/\p\to S/\p S$ is not bijective.
Applying Fact~\ref{fact130216a} to the modules $\Ext iMN$ for $i<\dim_R(N)$, we see that
$\p\notin\supp_R(\Ext iMN)$ for all  $i<\dim_R(N)$.
%, it suffices to show that
%there is an integer $i$ such that $0\leq i\leq\dim_R(N)-1$ and
%$\p\in\supp_R(\Ext iMN)$.
%Suppose that no such $i$ exists.
It follows that
$$\text{$0=\Ext iMN_{\p}\cong\Ext[R_{\p}]i{M_{\p}}{N_{\p}}$ \quad for all $i=0,\ldots,\dim_R(N)-1$.}$$
However, since $M_{\p}\neq 0\neq N_{\p}$, Lemma~\ref{lem130216a} implies that there is 
an integer $i_0\leq\depth_{R_{\p}}(N_{\p})\leq\dim_{R_{\p}}(N_{\p})\leq\dim_R(N)$ such that $\Ext[R_{\p}] {i_0}{M_{\p}}{N_{\p}}\neq 0$.
The displayed vanishing implies that 
$$\dim_R(N)\leq i_0\leq\dim_{R_{\p}}(N_{\p})\leq\dim_R(N)$$
so we have $\dim_{R_{\p}}(N_{\p})\leq\dim_R(N)$. It follows that $\p=\m$, so the induced map $R/\p\to S/\p S$ is an isomorphism
by assumption, contradicting our supposition on~$\p$.

\eqref{thm130216a5}$\implies$\eqref{thm130216a6}.
%$\eqref{thm130216a5}\implies\eqref{thm130216a6}$
Assume that $\Ext i{\Otimes SM}N$ is finitely generated over $R$ for all $i\geq 1$.
Remark~\ref{rmk130217a} implies $\Ext i{\Otimes SM}N$ is finitely generated over $R$ for all $i\geq 0$.
Let $J$ be an $R$-injective resolution of $N$, and let $P$ be an $R$-projective resolution of $M$.
To show that the natural map $\Ext i{\Otimes SM}N\to\Ext iMN$ is bijective for all $i\geq 0$,
it suffices to show that the natural chain map
$$f\colon\Hom{\Otimes SP}{J}\to\Hom PJ$$ 
is a quasiisomorphism.

Let $K=K^R(y_1,\ldots,y_e)$ denote the Koszul complex on a minimal generating sequence for $\m$.
Our assumption implies that the complexes $\Hom{\Otimes SP}{J}$ and $\Hom PJ$ have finitely generated homology over $R$.
Thus, Lemma~\ref{lem130217a} shows that it suffices to show that the following induced chain map
is a quasiisomorphism.
$$\Otimes Kf\colon\Otimes K\Hom{\Otimes SP}{J}\to\Otimes K\Hom PJ$$ 
The chain map $\Otimes Kf$ is compatible with the following (quasi)isomorphisms.
\begin{align*}
\Otimes K\Hom{\Otimes SP}{J}
&\cong\Otimes K\Hom{S}{\Hom PJ} \\
&\cong\Hom{\shift^{-e}\Otimes KS}{\Hom PJ} \\
&\simeq\Hom{\shift^{-e}K}{\Hom PJ} \\
&\cong\Otimes{K}{\Hom PJ}
\end{align*}
The first step in this sequence is adjointness.
The second and fourth steps follow from the fact that $K$ is a self-dual bounded complex of finite-rank free $R$-modules.
For the third step, the assumptions on $\vf$ imply that the chain map $K\to\Otimes KS$ is a quasiisomorphism (see~\cite[2.3]{frankild:amsveem});
since $\Hom PJ$ is a bounded-above complex of injective $R$-modules, the induced chain map
$$\Hom{\shift^{-e}\Otimes KS}{\Hom PJ} 
\xra\simeq\Hom{\shift^{-e}K}{\Hom PJ}$$
is also a quasiisomorphism.

\eqref{thm130216a3}$\implies$\eqref{thm130216a5}.
%$\eqref{thm130216a3}\implies\eqref{thm130216a5}$
Assume that $\Ext iMN$ has an $S$-module structure compatible with its $R$-module structure via $\vf$  for all $i\geq 0$.
To show that
$\Ext i{\Otimes SM}N$ is finitely generated over $R$ for all $i\geq 1$,
we show that $\Ext i{\Otimes SM}N\cong \Ext iMN$.
To this end, we use the spectral sequence
$$\Ext pS{\Ext qMN}\implies\Ext{p+q}{\Otimes SM}N.$$
If you like, this spectral sequence comes from the $R$-flatness of $S$ and the adjointness isomorphism 
$\Rhom S{\Rhom MN}\simeq\Rhom{\Lotimes SM}{N}$ in $\catd(R)$.
As each $\Ext qMN$ has a compatible $S$-module structure,
Fact~\ref{fact130216a} implies that $\Hom S{\Ext qMN}\cong\Ext qMN$
and that 
$\Ext pS{\Ext qMN}=0$ for all $p\geq 1$.
Hence, the spectral sequence degenerates, implying that
$$\Ext{q}{\Otimes SM}N\cong \Ext 0S{\Ext qMN}\cong\Ext qMN$$
as desired. \qed
\end{Proof}

The following example shows that the range $i=0,\ldots,\dim_R(N)-1$ is optimal in condition~\eqref{thm130216a4} of
Theorem~\ref{thm130216a}.

\begin{ex}\label{ex130217b}
Let $k$ be a field, and consider the localized polynomial ring $R=k[X_1,\ldots,X_d]_{(X_1,\ldots,X_d)}$ with $d\geq 1$.
Choose $j$  such that $0\leq j<d$, and set $M= R/(X_1,\ldots,X_{j})$ and $N=R/(X_{j+1},\ldots,X_{d-1})$.
(If $j=d-1$, this means that $N=R$.)
Note that $N$ does not have a compatible $\comp R$-module structure by Fact~\ref{fact130216a}.

The sequence $X_1,\ldots,X_{j}$ is 
$R$-regular, so the Koszul complex $K^R(X_1,\ldots,X_{j})$ is an $R$-free resolution of $M$.
The fact that $X_1,\ldots,X_{j}$ is 
$N$-regular implies that
$$\Ext iMN\cong
\begin{cases}0
&\text{for all $i\neq j=\dim(N)-1$}\\
N&\text{for all $i= j=\dim(N)-1$.}\end{cases}$$
In particular, the module $\Ext iMN$ has a compatible $\comp R$-module structure for $i=0,\ldots,\dim_R(N)-2$,
but $\Ext{\dim_R(N)-1}MN$ does not have a compatible $\comp R$-module structure.
\end{ex}

The final result of this paper uses the following definition.

\begin{defn}\label{defn130217a}
Let $R$ be a ring with Jacobson radical $J$.
Let
$\nak(R)$
denote the class of all $R$-modules $L$ such that either $L=0$ or $L/JL\neq 0$.
\end{defn}

\begin{rmk}\label{rmk130217b}
Let $R$ be a ring with Jacobson radical $J$.
Nakayama's Lemma implies that every finitely generated $R$-module is in $\nak(R)$.
When $R$ is local, the terminology ``$L$ satisfies NAK'' from~\cite{anderson:neams}
is equivalent to $L\in\nak(R)$.
\end{rmk}

\begin{rmk}\label{rmk130217c}
In~\cite{anderson:neams}, it is shown that the conditions of Fact~\ref{fact130216a} are equivalent to the following condition:
\begin{enumerate}[(ix)]
\item
$\Ext iSN$ is in $\nak(S)$ (equivalently, in $\nak(R)$) for  $i= 1,\ldots,\dim_R(N)$.
\end{enumerate}
The next result gives a version of this in our situation. (Note that the numbering is chosen to compare with the
conditions in Theorem~\ref{thm130216a}.)
\end{rmk}

\begin{prop}\label{prop130217a}
Let $\vf\colon (R,\m,k)\to (S,\n,l)$ be a flat local ring homomorphism such that the induced map $R/\m\to S/\m S$ is an isomorphism,
and let $M$ and $N$ be finitely generated $R$-modules.
Consider the following conditions:
\begin{enumerate}[\rm(i)]
\item[\rm(v)] \label{prop130217a1}
$\Ext i{\Otimes SM}N$ is finitely generated over $R$ for all $i\geq 1$.
\item[\rm(v')] \label{prop130217a1}
$\Ext i{\Otimes SM}N=0$ for all $i\geq 1$.
\item[\rm(viii)] \label{prop130217a2}
$\Ext i{\Otimes SM}N$ is in $\nak(R)$ (i.e., in $\nak(S)$) for $i= 1,\ldots,\dim_R(N)$.
\end{enumerate}
The implications \emph{(v')}$\implies$\emph{(v)}$\implies$\emph{(viii)} always hold,
and the three conditions are equivalent when $\Ext iMN=0$ for all $i\geq 1$.
\end{prop}

\begin{proof}
Note that, for any $S$-module $L$, one has $L/\m L=L/\n L$ since $\n=\m S$. Thus,
one has $L\in\nak(R)$ if and only if $L\in\nak(S)$.
Also, the implication (v')$\implies$(v) is trivial.

(v)$\implies$(viii).
Assume that $\Ext i{\Otimes SM}N$ is finitely generated over $R$. Since $\Ext i{\Otimes SM}N$ has a compatible
$S$-module structure, the previous paragraph and Remark~\ref{rmk130217b} imply that $\Ext i{\Otimes SM}N$ is in $\nak(R)$ and $\nak(S)$.

(viii)$\implies$(v).
Assume that $\Ext iMN=0$ for all $i\geq 1$ and 
$\Ext i{\Otimes SM}N$ is in $\nak(R)$ for $i= 1,\ldots,\dim_R(N)$.
The vanishing of $\Ext iMN$ and the flatness of $S$ imply that we have the following isomorphisms in $\catd(R)$.
\begin{align*}
\Rhom {\Otimes SM}{N}
&\simeq\Rhom {\Lotimes SM}{N}\\
&\simeq\Rhom{S}{\Rhom MN}\\
&\simeq\Rhom{S}{\Hom MN}
\end{align*}
Taking homology, we conclude that
$$\Ext iS{\Hom MN}\cong\Ext i{\Otimes SM}N \in \nak(R)$$
for $i= 1,\ldots,\dim_R(N)$. By Remark~\ref{rmk130217c}, we conclude that for $i\geq 1$ we have
%$\Hom MN$ has a compatible $S$-module structure
$\Ext i{\Otimes SM}N\cong\Ext iS{\Hom MN}=0$,
as desired.
\end{proof}

%\section*{Acknowledgments}

%\bibliography{../+new}

\begin{thebibliography}{1}

\bibitem{anderson:neams}
B.~J. Anderson and S.~Sather-Wagstaff, \emph{{NAK} for {E}xt and ascent of
  module structures}, Proc. Amer. Math. Soc, to appear,
  \texttt{arXiv:1201.3039}.

\bibitem{christensen:gd}
L.~W. Christensen, \emph{Gorenstein dimensions}, Lecture Notes in Mathematics,
  vol. 1747, Springer-Verlag, Berlin, 2000. \MR{2002e:13032}

\bibitem{frankild:dcev}
A.~J. Frankild and S.\ Sather-Wagstaff, \emph{Detecting completeness from
  {E}xt-vanishing}, Proc. Amer. Math. Soc. \textbf{136} (2008), no.~7,
  2303--2312. \MR{2390496 (2009c:13016)}

\bibitem{frankild:amsveem}
A.~J. Frankild, S.\ Sather-Wagstaff, and R.~A. Wiegand, \emph{Ascent of module
  structures, vanishing of {E}xt, and extended modules}, Michigan Math. J.
  \textbf{57} (2008), 321--337, Special volume in honor of Melvin Hochster.
  \MR{2492456}

\bibitem{hartshorne:rad}
R.\ Hartshorne, \emph{Residues and duality}, Lecture Notes in Mathematics, No.
  20, Springer-Verlag, Berlin, 1966. \MR{36 \#5145}

\bibitem{verdier:cd}
J.-L.\ Verdier, \emph{Cat\'{e}gories d\'{e}riv\'{e}es}, SGA 4$\frac{1}{2}$,
  Springer-Verlag, Berlin, 1977, Lecture Notes in Mathematics, Vol. 569,
  pp.~262--311. \MR{57 \#3132}

\bibitem{verdier:1}
\bysame, \emph{Des cat\'egories d\'eriv\'ees des cat\'egories ab\'eliennes},
  Ast\'erisque (1996), no.~239, xii+253 pp. (1997), With a preface by Luc
  Illusie, Edited and with a note by Georges Maltsiniotis. \MR{98c:18007}

\end{thebibliography}
\providecommand{\bysame}{\leavevmode\hbox to3em{\hrulefill}\thinspace}
\providecommand{\MR}{\relax\ifhmode\unskip\space\fi MR }
% \MRhref is called by the amsart/book/proc definition of \MR.
\providecommand{\MRhref}[2]{%
  \href{http://www.ams.org/mathscinet-getitem?mr=#1}{#2}
}
\providecommand{\href}[2]{#2}

\end{document}